\documentclass[11pt,reqno]{amsart}
\usepackage{fullpage}

\usepackage{amsmath}
\usepackage{amsfonts}
\usepackage{amsthm}
\usepackage{amssymb}
\usepackage{enumerate}
\usepackage{graphicx}
\usepackage{subfigure}
\usepackage{mathrsfs}
\usepackage{multicol}

\usepackage{hyperref}

\theoremstyle{plain}
\newtheorem{thm}{Theorem}[section]
\newtheorem{theorem}{Theorem}[section]

\newtheorem{lemma}[thm]{Lemma}

\newtheorem{corollary}[thm]{Corollary}

\newtheorem{proposition}[thm]{Proposition}
\theoremstyle{remark}

\numberwithin{equation}{section}

\newcommand{\E}{\mathbb{E}}
\newcommand{\M}{\mathbb{M}}
\newcommand{\N}{\mathbb{N}}
\newcommand{\bP}{\mathbb{P}}

\newcommand{\R}{\mathbb{R}}

\newcommand{\Z}{\mathbb{Z}}

\newcommand{\cB}{\mathcal{B}}
\newcommand{\bB}{\mathbf{B}}
\newcommand{\bD}{\mathbf{D}}

\newcommand{\cI}{\mathcal{I}}
\newcommand{\cL}{\mathscr{L}}

\newcommand{\cS}{\mathcal{S}}

\newcommand{\cE}{\mathcal{E}}
\newcommand{\cF}{\mathcal{F}}

\newcommand{\Q}{\mathbb{Q}}

\newcommand{\gam}{\gamma}
\newcommand{\del}{\delta}
\newcommand{\ep}{\varepsilon}

\newcommand{\lam}{\lambda}

\newcommand{\om}{\omega}

\newcommand{\Om}{\Omega}

\newcommand{\ol}{\overline}

\newcommand{\Div}{{\rm div}\,}

\newcommand{\tr}{{\rm tr}\,}

\newcommand{\epi}{{\rm epi}\,}

\author[W. Jing]{Wenjia Jing}
\address{Department of Mathematics, 
	      The University of Chicago, 
	      5734 S. University Avenue Chicago, 
	      IL 60637, USA}
\email{wjing@math.uchicago.edu}

\author[P.E. Souganidis]{Panagiotis E. Souganidis}
\address{Department of Mathematics, 
	      The University of Chicago, 
	      5734 S. University Avenue Chicago, 
	      IL 60637, USA}
\email{souganidis@math.uchicago.edu}

\author[H.V. Tran]{Hung V. Tran}
\address{Department of Mathematics, 
	      University of Wisconsin Madison, 
	      Van Vleck Hall, 480 Lincoln Drive, 
	      WI 53706, USA}
\email{hung@math.wisc.edu}
	           
\title[Homogenization of HJ equations in dynamic random environment]
{Stochastic homogenization of viscous superquadratic Hamilton-Jacobi equations in dynamic random environment}

\begin{document}

\begin{abstract}

We study the qualitative homogenization of second order  viscous Hamilton-Jacobi equations in space-time stationary ergodic random environments. Assuming that the Hamiltonian is convex and superquadratic in the momentum variable (gradient) we establish a homogenization result and characterize the effective Hamiltonian for arbitrary (possibly degenerate) elliptic diffusion matrices. The result extends previous work that required uniform ellipticity and space-time homogeneity for the diffusion.

\medskip


\medskip

\end{abstract}

\maketitle

\section{Introduction}\label{sec:Intro}

We study the homogenized behavior of the solution $u^\ep = u^\ep(x,t,\om)$ to the second order (viscous) Hamilton-Jacobi equation
\begin{equation}\label{PDE1}
\left\{
\begin{aligned}
&u^\ep_t - \ep \tr \left(A\left(\frac{x}{\ep},\frac{t}{\ep}, \om\right) D^2 u^\ep \right) + H\left(Du^\ep, \frac{x}{\ep},\frac{t}{\ep},\omega\right)=0 &\quad &\text{in} \ \R^n \times (0,+\infty),\\
&u^\ep =u_0 &\quad &\text{on} \ \R^n \times \{0\}, 
\end{aligned}
\right.
\end{equation}
where $u_0 \in {\rm BUC}(\R^n)$, the space of bounded uniformly continuous functions in $\R^n$, and, for each element  $\om$ of the underlying probability space $(\Om, \cF, \bP)$, the diffusion matrix $A = (a_{ij}(x,t,\omega))$ is elliptic, possibly degenerate, and, for all $x,t$ and $\om$, the Hamiltonian $H=H(p,x,t,\om)$ is convex and has superquadratic growth in $p$. Moreover, $A(\cdot,\cdot,\om)$ and $H(p,\cdot,\cdot,\om)$ are stationary ergodic random fields on $(\Om,\cF,\bP)$. The precise assumptions are detailed in Section \ref{sec:prelim}.

The standard viscosity solution theory yields that, for each $\om \in \Om$, \eqref{PDE1} is well posed. The homogenization result is that there exists an effective Hamiltonian $\ol{H}: \R^n \to \R$ such that, if $\ol{u}$ is the unique solution to the homogenized Hamilton-Jacobi equation
\begin{equation}\label{PDE-hom}
\left\{
\begin{aligned}
&  \ol{u}_t + \overline H(D\overline u) = 0 &\quad &\text{in} \ \R^n \times (0,\infty),\\[1mm]
&\overline u = u_0  &\quad &\text{on} \ \R^n \times \{0\},
\end{aligned}
\right.
\end{equation}
then, for almost every $\om \in \Om$, the solution $u^\ep$ to \eqref{PDE1} converges locally uniformly to $\ol{u}$, that is there exists an event $\widetilde \Om \in \cF$ with full measure such that, for every $\om \in \widetilde\Om$, $R>0$ and $T > 0$,
\begin{equation}
\label{eq:unifcov}
\lim_{\ep \to 0} \sup_{|x| \le R, t \in [0,T]} \, \left| u^\ep(x,t,\om) - \ol{u}(x,t) \right| = 0.
\end{equation}

The ``viscous'' Hamilton-Jacobi equation \eqref{PDE1} arises naturally in the study of large deviations of diffusion process in spatiotemporal random media, in which case $H$ is quadratic in the gradient. It also finds applications in stochastic optimal control theory; we refer to Fleming and Soner \cite{FS:06} for more details. The homogenization result above serves as a model reduction in the setting when the environment is highly oscillatory but, nevertheless, satisfies certain self-averaging properties. In particular, when the diffusion matrix in the underlying stochastic differential equation depends on time, the coefficient $A$ in \eqref{PDE1} will be time dependent as well. As far as we know and as argued below, the homogenization problem in this setting has been open.

The periodic homogenization of coercive Hamilton-Jacobi equations was first studied by Lions, Papanicolaou and Varadhan \cite{LPV} and, later, Evans \cite{Ev:89,Ev:92} and Majda and Souganidis \cite{MS:94}. Ishii proved in \cite{I:00} homogenization in almost periodic settings. The stochastic homogenization of first order Hamilton-Jacobi equations was established independently by Souganidis \cite{S} and Rezakhanlou and Tarver \cite{RT}. Later  Lions and Souganidis \cite{LS2} and Kosygina, Rezakhandou and Varadhan \cite{KRV} proved independently stochastic homogenization for viscous Hamilton-Jacobi equations using different methods and complementary assumptions. In \cite{LS3} Lions and Souganidis gave a simpler proof for homogenization in probability using weak convergence techniques. This program was extended by Armstrong and Souganidis in \cite{AS1,AS2} where the metric-problem approach was introduced. Some of the results of \cite{AS1,AS2} were revisited by Armstrong and Tran in \cite{AT1}. All of the aforementioned works in random homogenization required the Hamiltonian $H$ to be convex. The homogenization of general nonconvex Hamiltonians in random environments remains to date an open problem. A first extension to level-set convex Hamiltonians was shown  by Armstrong and Souganidis in \cite{AS2}. Later, Armstrong, Tran and Yu \cite{ATY1,ATY2} proved stochastic homogenization for separated Hamiltonians of the form $H=h(p)-V(x,\om)$ with general non-convex $h$ and random potential $V(x,\om)$ in one dimension. Their methods also established homogenization of some special non-convex Hamiltonians in all dimensions. Armstrong and Cardaliaguet \cite{AC} studied the homogenization of positively homogeneous non convex Hamilton-Jacobi equations in strongly mixing environments. More recently Feldman and Souganidis \cite{FS:162} established  homogenization of strictly star-shaped Hamiltonians in similar environments.  Ziliotto \cite{Zil:16} constructed an example of a non-convex separated Hamiltonian in two dimensions that does not homogenize. 
Feldman and Souganidis \cite{FS:16} extended the construction to any separated $H$ that has a strict saddle point. In addition  \cite{FS:16} also yields non-convex Hamiltonians with space-time random potentials for which the Hamilton-Jacobi equation does not homogenize.

The aforementioned PDE approaches for stochastic homogenization, that is the weak convergence technique and the metric problem approach, were developed for random environments that are time independent. In this setting, one has uniform in $\ep$ and $\om$ Lipschitz estimates for $u^\ep(\cdot,\om)$, which, however, are not available if $A$ and $H$ depend on $t$. Nevertheless, Kosygina and Varadhan \cite{KV} established homogenization results for viscous Hamilton-Jacobi equations with constant diffusion coefficients, more precisely $A$ being the identity matrix, using the stochastic control formula and invariant measures. For first order equations with superlinear  Hamiltonians, homogenization results were proved by Schwab \cite{Sch}. Recently, the authors \cite{JST} established homogenization for linearly growing Hamiltonians that are periodic in space and stationary ergodic in time.

In this paper, we extend and combine the methodologies of \cite{LS2} and \cite{AS1,AS2,AT1} and obtain stochastic homogenization for general viscous Hamilton-Jacobi equations in dynamic random environments. The results of \cite{LS2} was based on the analysis of a special solution to \eqref{PDE1} that we loosely call the fundamental solution. This is a sub-additive, stationary function which, in view of the subadditive ergodic theorem, has a homogenized limit, that identifies the convex conjugate of the effective Hamiltonian $\ol{H}$. At the $\ep >0$ level, however, the fundamental solution gives rise only to supersolutions $v^\ep$ to \eqref{PDE-hom}. One of the key steps in \cite{LS2} was to show that the difference between $u^\ep$ and $v^\ep$ tends to $0$ as $\ep \to 0$. This made very strong use of the uniform Lipschitz estimates on $u^\ep$ which were also  proved there and are not available for time dependent problems. The methodology of \cite{AS1,AS2} was based on the analysis of the solution to the metric problem which loosely speaking is the ``minimal cost'' to connect two points. The metric solution  is a sub-additive stationary function and has a homogenized limit, which, at each level, is the support function of the level set of the effective $\ol{H}$. The homogenization result was then proved in \cite{AS1,AS2} by developing a reversed perturbed test function argument. In the dynamic random setting, however, the ``metric" between two  points in space  must depend on a starting time and, hence is not suitable for such environments.

Here we use the  fundamental solution  approach of \cite{LS2} to find the effective Hamiltonian and the reverse perturbed test function method of \cite{AS1,AS2} to establish the homogenization. The main contribution of the paper  is to ``go away'' from the need to have uniform in $\ep$ Lipshitz bounds. Indeed  the uniform convergence of the fundamental solution to its homogenized limit uses only a uniform (in $\ep$) modulus of continuity in its first pair of variables, which is available for superquadratic Hamilton-Jacobi equations \cite{CC,CS:12}. Similarly the reverse perturbed test function argument is adapted to work without the need of Lipschitz bounds.


We summarize next the main results of this paper. For each fixed $\om \in \Om$, let $\cL(x,t,y,s,\om)$ denote the fundamental solution of \eqref{PDE1}; see \eqref{eq:fundsol} below. The first result  is that $\cL(x,t,y,s,\om)$ has long time average, that is there exists a convex function $\ol{L} : \R^n \to \R$, known as  the effective Lagrangian, such that, for a.s. $\omega \in \Omega$ and locally uniformly in $(x,t)$ for $t>0$, 
\begin{equation}
\label{eq:homogL}
\lim_{\rho \to \infty} \frac{1}{\rho} \cL(\rho x, \rho t, 0,0,\omega) = t\,\ol{L}\left(\frac{x}{t}\right).
\end{equation}
We note that although the pointwise convergence for fixed $(x,t)$ is a direct consequence of subadditive ergodic theorem, the locally uniform convergence requires some uniform (in $\om$ and $\rho$) continuity of the scaled function $\rho^{-1} \cL(\rho\,\cdot, \rho\,\cdot, 0,0,\om)$. This is where the superquadratic growth of $H$ is used. Indeed, under this assumption, Cannarsa and Cardaliaguet \cite{CC}, and Cardaliaguet and Silvestre \cite{CS:12} obtained space-time $C^{0,\alpha}$-estimates for bounded solutions, 
which depend on the growth condition of $H$ but neither  on the ellipticity of diffusion matrix $A$ nor on the smoothness of $H$ or $A$. Here we obtain the desired continuity by applying these regularity results to the scaled fundamental solutions.

The effective Hamiltonian is then defined by
\begin{equation}
\label{e.Hbar}
\ol{H}(p) = \sup_{v \in \R^n} \left ( p\cdot v - \ol{L}(v) \right ),
\end{equation}
and the homogenized equation is precisely \eqref{PDE-hom}.  Then we show that $\ol{H}$ is also the limit of the solutions to the approximate cell problem, a fact which yields the homogenization for the general equation.


The rest of the paper is organized as follows. In the remaining part of the introduction we review most of the standard notation used in the paper. In the next section, we introduce the precise assumptions and state the main results. In section \ref{sec.L} we prove \eqref{eq:homogL}. In section \ref{s.lim=H}, we show that the effective $\ol{H}$ defined in \eqref{e.Hbar} agrees with the uniform limit of the solution to the approximate cell problem. The homogenization result for the Cauchy problem \eqref{PDE1} follows from this fact. In section \ref{sec:formulae} we show that, as a consequence of the homogenization result proved in this paper, the effective Hamiltonian is given by formulae similar to the ones established in \cite{LS2} and \cite{KV}. 

\subsection*{Notations}
We work in the $n$-dimensional Euclidean space $\R^n$. The subset of points with rational coordinates is denoted by $\Q^n$. The open ball in $\R^n$ centered at $x$ with radius $r > 0$ is denoted by $B_r(x)$, and this notation is further simplified to $B_r$ if the center is the origin. The volume of a measurable set $A \subseteq \R^n$ is denoted by $\mathrm{Vol}(A)$. The $n+1$ dimensional space-time is denoted by $\R^n \times \R$ or simply by $\R^{n+1}$. The space time cylinder of horizontal radius $R>0$ and vertical interval $(r,\rho)$ centered at a space-time point $(x,t)$ is denoted by $Q_{R,r,\rho}(x,t)$, that is  $Q_{R,r,\rho}(x,t) = \{(y,s) \,:\, y \in B_R(x), s \in (t+r,t+\rho)\}$; to further simplify notations, we omit the reference point $(x,t)$ when it is $(0,0)$. Moreover, $Q_R$ is a short-hand notation for the cylinder $Q_{R,-R,R}$. For two vectors $u, v \in \R^n$, $\langle u, v\rangle$ denotes the inner product between $u$ and $v$, and  $\M^{n\times m}$ denotes the set of $n$ by $m$ matrices with real entries, and $\M^n$ is a short-hand notation of $\M^{n \times n}$. The identity matrix is denoted by $Id$. Finally, $\mathcal B(\Xi)$ denotes the Borel $\sigma$-algebra of the metric space $\Xi$. 


\section{Assumptions, the fundamental solution, and the main results}
\label{sec:prelim}

\subsection{The general setting and assumptions}

We consider a probability space $(\Omega, \cF, \bP)$ endowed with an {ergodic group of measure preserving transformations} $\{\tau_{(x,t)} \,:\, (x,t)\in\R^{n+1}\}$, that is, a family of maps $\tau_{(x,t)}:\Omega\to\Omega$ satisfying, for all $(x,t), (x',t')\in \R^{n+1}$ and all $E \in\cF$,
\begin{enumerate}
\item[(P1)] $\tau_{(x+x',t+t')}=\tau_{(x,t)}\circ\tau_{(x',t')} \, \hbox{ and }\;
	\bP[\tau_{(x,t)} E] = \bP[E]$, 
\end{enumerate}
and 
\begin{enumerate}
\item[(P2)] $\tau_{(y,s)}(E)=E \ \text{for all} \  (y,s) \in\R^{n+1}
	 \hbox{ implies }\bP[E] \in \{0,1\}$.
\end{enumerate}

The diffusion matrix $A = (a_{ij}(x,t,\omega)) \in \M^n$ is given by $$A = \sigma \sigma^T$$ where $\sigma = \sigma(x,t,\omega)$ is  an $M^{n\times m}$ 
valued random process. 

As far as $H:\R^n \times \R^n \times \R \times\Omega\to\R$ and $\sigma: \R^n\times \R \times \Om \to \M^{n\times m}$ are concerned, we assume henceforth that 
\begin{enumerate}
\item[(A1)]  $H$ and $\sigma$ are $\cB(\R^{n}\times\R^n \times \R)\times\cF$ and $\cB(\R^n \times \R) \times \cF$ measurable respectively,
 \item[(A2)] for any fixed $p \in \R^d$, $\sigma$ and $H$ are stationary in $x$ and $t$, 
that is, for every $(x,t)\in \R^{n+1}$, $(y,s)\in \R^{n+1}$ and $\omega\in\Omega$, 
\begin{equation*}
\sigma(x+y,t+s,\omega) = \sigma(x,t,\tau_{(y,s)}\omega) \ \text{and} \ H(p, x+y,t+s,\omega) = H(p,x,t,\tau_{(y,s)}\omega),
\end{equation*}
\item[(A3)] for each $p \in \R^n$ and $\om\in \Om$, $\sigma(\cdot,\cdot,\omega)$ and $H(p,\cdot,\cdot,\om)$ are Lipschitz continuous in $x$ and $t$,
 \item[(A4)] there exists $\gam > 2$ and $C \ge 1$ such that, for all $(x,t) \in \R^{n+1}, \omega\in\Omega$ and  $p \in \R^n$,
\begin{equation}\label{A2}
\frac{1}{C} |p|^\gam - C \le H(p,x,t,\om) \le C(|p|^\gam + 1),
\end{equation}
\end{enumerate}
and, finally,
\begin{enumerate}
\item[(A5)] the mapping $p\mapsto H(p,x,t,\omega)$ is convex for all $(x,t,\om) \in \R^{n+1} \times \Omega$.
\end{enumerate}

Since throughout the paper we use all the above assumptions, we summarize them as
\begin{enumerate}
\item[(A)] assumptions (P1), (P2), (A1). (A2), (A3), (A4) and (A5) hold.
\end{enumerate}

\subsection{The fundamental solution} For each $\om \in \Om$ and $(y,s) \in \R^n \times \R$, we define the fundamental solution  $\cL := \cL(\cdot,\cdot,y,s,\om) : \R^n \times (s,\infty) \to \R$ to be the unique viscosity solution to
\begin{equation}
\label{eq:fundsol}
\left\{
\begin{aligned}
&\cL_t - \tr(A(\cdot,\cdot,\om)D^2 \cL) + H(D \cL,\cdot,\cdot,\om) = 0 &\quad &\mbox{in} \ \R^n \times (s,\infty),\\
&\cL(\cdot,s,y,s,\om) = \delta(\cdot,y) &\quad &\mbox{in} \ \R^n,
\end{aligned}
\right.
\end{equation}
where 
$ \delta(x,y)=0$ if $x=y$ and $ \delta(x,y)=\infty$ in $\R^n \setminus \{y\}.$
As in Crandall, Lions and Souganidis \cite{CLS:89}, this initial condition is understood in the sense that $\cL(\cdot,t,y,s,\om)$ converges, as $t$ decreases to $s$, locally uniformly on $\R^n$ to the function $\delta(\cdot,y)$. The existence and uniqueness of $\cL$ follows from an almost straight forward modification of the results of \cite{CLS:89}. In view of the stochastic control representation of Hamilton-Jacobi equations, $\cL(x,t,y,s,\om)$ is the ``minimal cost'' for a controlled diffusion process in the random environment determined by $(\sigma,H)$ to reach the vertex $(y,s)$ from $(x,t)$.  

\subsection{Main theorems}
The first result is about the long time behavior of the fundamental solution which yields the effective Lagrangian $\ol{L}$. The proof  is given  at the end of Section \ref{sec.L}.

\begin{theorem}\label{t.L.homog} 
Assume {\upshape(A)}. There exist $\widetilde\Om \in \cF$ with $\bP(\widetilde\Om) = 1$ and a convex function $\ol{L} : \R^n \to \R$ such that, for all $\om \in \widetilde\Om$, $r>0$ and $R>r$,
\begin{equation}\label{e.L.homog}
\lim_{\rho \to \infty} \sup_{(y,s)\in Q_R} \, \sup_{(x,t) \in Q_{R,r,R}((y,s))} \left\lvert \frac{1}{\rho} \cL(\rho x, \rho t, \rho y, \rho s, \om) - (t-s) \ol{L}\left(\frac{x-y}{t-s}\right) \right\rvert  = 0.
\end{equation}
\end{theorem}




Let $\overline u$ be the solution to \eqref{PDE-hom}, with the effective Hamiltonian $\ol{H}$ is defined  \eqref{e.Hbar}  and is, hence,  the Legendre transform of the effective Lagrangian $\ol{L}$.

The homogenization result is stated next.

\begin{theorem}\label{thm:stoch}
Assume {\upshape(A)} and let $\widetilde\Omega$ be as in Theorem \ref{t.L.homog}. For each $\omega\in\widetilde\Omega$, the solution $u^\ep$ of \eqref{PDE1} converges, as $\ep \to 0$ and locally uniformly in $\R^n \times [0,\infty)$, to $\overline u$.
\end{theorem}

It is well known that the Theorem \ref{thm:stoch} follows from variations of the the perturbed test function method \cite{Ev:89} if, for each $p \in \R^n$, the solution $w^\ep$ of the approximate auxiliary (cell) problem 
\begin{equation}
\label{eq:appcell}
\ep w^\ep(x,t) + w^\ep_t(x,t) - \tr(A(x,t,\om)D^2 w^\ep(x,t)) + H(p+Dw^\ep,x,t) = 0 \quad \text{ in } \R^n \times \R
\end{equation}
satisfies $\ep w^\ep$ converge uniformly to $-\ol{H}(p)$ in cylinders of radius $\sim 1/\ep$ as $\ep \to 0$. In the classical periodic setting the convergence is uniform. The need to consider large sets varying with $\ep$ was first identified in \cite{S}. Because this is standard, we omit the proof and refer, for example, to  \cite[section 7.3]{AS1} for the complete argument.

For all $\ep>0$ and $\om \in \Om$, the approximate cell problem \eqref{eq:appcell} is well posed. Recall that $Q_R \subseteq \R^{n+1}$, $R>0$, is the cylinder centered at $(0,0)$ with radius $R$. The precise statement about the convergence of $\ep w^\ep$ to $-\ol{H}(p)$ is stated in the next Theorem.

\begin{theorem}\label{thm:appcell} Assume {\upshape(A)} and let $\widetilde\Om$ be as in Theorem \ref{t.L.homog}. Then, for all $\om \in \widetilde\Om$, $p\in\R^n$ and $R>0$,
\begin{equation}
\label{eq:weplim}
\lim_{\ep \to 0} \, \sup_{Q_{R/\ep}} \, \left| \ep w^\ep(\cdot,\cdot,\om,p) + \ol{H}(p)\right| = 0.
\end{equation}
\end{theorem}
The proof of Theorem \ref{thm:appcell}, which is given in Section \ref{s.lim=H}, is based on reversed test function argument of \cite{AS1,AT2}. The differences here are the lack of Lipshitz bounds and the need to apply the method to the scaled versions of the fundamental solution instead of the metric solution.

\section{The long time behavior of the fundamental solution}
\label{sec.L}

We investigate the long time average 
of the fundamental solution $\cL$, as $\rho \to \infty$. The averaged function is given by the subadditive ergodic theorem, which is a natural tool for the study 
of $\cL$ in view of the following Lemma. 

\begin{lemma}\label{lem:statsubadd}
Assume {\upshape(A)}. Then, for all $\om \in \Om$ and $x,y,z\in \R^d$,

\noindent{\upshape(i)} if $t,s,\rho \in \R$ and $t \ge s$, then
\begin{equation}\label{e.L.stationary}
\cL(x+z,t+\rho,y+z,s+\rho,\om) = \cL(x,t,y,s,\tau_{(z,\rho)}\om), 
\end{equation}
and
 
\noindent{\upshape(ii)} if $t, s, r \in \R$ satisfy $s \le r \le t$, then
\begin{equation}\label{mono-t}
\cL(x,t,y,s,\om) \le \cL(x,t,z,r,\om) + \cL(z,r,y,s,\om).
\end{equation}
\end{lemma}

The stationarity of $\cL$ is an immediate consequence of the uniqueness of \eqref{eq:fundsol} and the stationarity of the environment. The subadditivity of $\cL$ follows from the comparison principle for \eqref{eq:fundsol} and the singular initial conditions of the fundamental solutions. Since the proof of Lemma \ref{lem:statsubadd} is standard, we omitted it. 

Next we recall a result of \cite[Proposition 6.9]{LS2} that concerns bounds on the unscaled function $\cL$. Although \cite{LS2} considered time homogeneous environments, the proof of the following result does not depend on that fact.

\begin{lemma}\label{l.L.bdd} 
Assume {\upshape(A)} and let $\gam' := \gam/(\gam-1)$. 
There exists a constant $C>0$ such that, for all $\om \in \Om$, $x,y \in \R^n$ and $t, s \in \R$ with $t > s$, 
\begin{equation}\label{eq:Lbdd}
-C(t-s) \le \cL(x,t,y,s,\omega) \le C\left( \frac{|x-y|^{\gam'}}{(t-s)^{\gam'-1}} + (t-s)^{1-\frac{\gam'}{2}} + (t-s)\right).
\end{equation}
\end{lemma}

To study the long time average of $\cL$, we define , for $\ep > 0$, the rescaled function
\begin{equation}
\label{eq:Lep}
\cL^\ep(x,t,y,s,\om) := \ep \cL\left( \frac{x}{\ep},\frac{t}{\ep},\frac{y}{\ep},\frac{s}{\ep},\om\right).
\end{equation}
It is immediate that, for each fixed $(y,s) \in \R^{n+1}$,  $\cL^\ep(\cdot,\cdot,y,s,\om)$ solves
\begin{equation}
\label{eq:cLe}
\left\{
\begin{aligned}
& \cL^\ep_t - \ep \tr\left(A\left(\frac{\cdot}{\ep},\frac{\cdot}{\ep},\om\right)D^2\cL^\ep\right) + H\left(D \cL^\ep, \frac{\cdot}{\ep},\frac{\cdot}{\ep},\om\right) = 0 &\quad &\text{ in } \R^n \times (s,\infty),\\
&\cL^\ep(\cdot,s,y,s,\om) = \ep\delta\left(\frac{\cdot}{\ep},\frac{y}{\ep}\right) = \delta(\cdot,y) &\quad &\text{ on } \R^n \times \{s\}.
\end{aligned}
\right.
\end{equation} 

It now follows from Lemma \ref{l.L.bdd}, after the rescaling, that, for all $t > s$,
\begin{equation*}
-C(t-s) \le \cL^\ep(x,t,y,s,\om) \le C\left(\frac{|x-y|^{\gam'}}{(t-s)^{\gam'-1}} + \ep^{\frac{\gam'}{2}}(t-s)^{1-\frac{\gam'}{2}} + (t-s)\right).
\end{equation*}
Note that $\gam' \in (1,2)$.  As a result, for all $0 < \ep < 1$, $R \ge 1$, $r \in (0,1)$, $x,y \in B_R$, and $t,s \in \R$ with $ t-s \in (r, R)$, we have
\begin{equation}
\label{e.Lep_bdd}
|\cL^\ep(x,t,y,s,\om)| \le CR\left(\frac{R}{r} \right)^{\gam' - 1} + CR,
\end{equation}
and, hence, $\cL^\ep$ is uniformly bounded on the set $\{(x,y,t,s) \,:\, x,y \in B_R, \, r \le t-s \le R\}$. This and the superquadratic growth of $H$ allow us to apply the H\"older regularity results in \cite{CC,CS:12} to get the following uniform in $\ep$ estimates for $\cL^\ep$.

\begin{proposition}\label{prop:uc} Assume {\upshape(A)}. Then there exists $\alpha \in (0,1)$, and, for all $R \ge 1$, $r \in (0,1)$, and $(y,s) \in \R^{n+1}$, the function $\cL^\ep(\cdot,\cdot,y,s,\om)$ is uniformly with respect to $\ep$, $\om$ and $(y,s)$ $\alpha$-H\"older continuous on the set $Q_{R,r,R}(y,s)$. 
\end{proposition}

We omit the proof of Proposition \ref{prop:uc}, which, in view of \eqref{e.Lep_bdd}, is an immediate consequence of Theorem 6.7 of \cite{CC} (see also Theorem 1.3 of \cite{CS:12}). 

\subsection{Long time average of $\cL$} 
The stationarity of $\cL$ in \eqref{e.L.stationary} and the scaling in the definition of $\cL^\ep$ suggest that the limit, as $\ep \to 0$,  of $\cL^\ep(x,t,y,s,\om)$ must only depend on $(x-y)/(t-s)$. To get the limit, it suffices to set $(y,s) = (0,0)$, $t = 1 > s$, and study the limit of the function $\rho^{-1}\cL(\rho x,\rho,0,0,\om)$ as $\rho \to \infty$. This is possible using the subadditive ergodic theorem, which yields a random variable $\ol{L}(x,\omega)$.


\begin{theorem}\label{thm:sae} Assume {\upshape(A)}. For any $x \in \R^n$, there exists a random variable $\ol{L}(x,\om)$ and $\Om_{x} \in \cF$ of full measure such that, for all $\om \in \Om_{x}$,
\begin{equation}\label{e.Llim.pw}
\lim_{\rho \to \infty} \frac{1}{\rho} \cL(\rho x, \rho, 0, 0, \om) = \ol{L}(x,\om).
\end{equation}
Moreover, $\ol{L}(x,\cdot)$ is almost surely the  constant $\E \ol{L}(x,\cdot).$
\end{theorem}


That $\ol{L}(x,\cdot)$ is deterministic is important for the final homogenization result. This is usually proved by showing that $\ol{L}(x,\cdot)$ is invariant with respect to the translations $\{\tau_{(y,s)}\}_{(y,s)\in \R^{n+1}}$. In the time homogeneous setting \cite{LS2} or for first order equations \cite{Sch}, the translation invariance of $\ol{L}(x,\cdot)$ is a consequence of uniform in $\ep$ continuity of $\cL^\ep(x,t,y,s,\om)$ in all of its variables. For the problem at hand 
Proposition \ref{prop:uc} gives that $\cL^\ep$ is uniformly  continuous with respect to its first pair of variables.  The uniform continuity with respect to the second  pair of variables, that is the vertex, is more subtle and unknown up to now. 

We prove next that $\ol{L}(x,\cdot)$ is translation invariant  without using uniform continuity of $\cL^\ep$ with respect to $(y,s)$. The argument  is based on two observations. Firstly, $\ol{L}(x,\cdot)$ is invariant when the vertex varies along the line $l_x := \{(tx,t)\,:\, t \in \R\}$. Secondly, the subadditive property \eqref{eq:Lbdd} and the uniform bounds \eqref{e.Lep_bdd} yield one-sided bounds for $\cL$. Indeed, to bound $\cL(\cdot,\cdot,y,s,\om)$ from above, we compare it with $\cL(\cdot,\cdot,z,r)$ at a vertex $(z,r)$ such that $r>s$, and bounded $|r-s|$ and $|z-y|$. Similarly, for a lower bound, we compare with a vertex that has $r<s$.

The proof of Theorem \ref{thm:sae} is divided in three steps. In the first step  we identify $\ol{L}(x,\om)$ by applying the subadditive ergodic theorem to $\rho^{-1} \cL(\rho x, \rho, kx,k,\om)$ with vertex $(kx,k) \in l_x$. Then, we establish the invariance of $\ol{L}(x,\om)$ with respect to vertices in $l_x$. Finally in the third step, we show that $\ol{L}(x,\cdot)$ is invariant with respect to $\{\tau_{(y,s)}\}$ and, hence, deterministic. 

\begin{proof}[Proof of Theorem \ref{thm:sae}] {\itshape Step 1: The convergence of $\cL$ with vertex $(0,0)$}. This is a straight forward application of the classical subadditive ergodic theorem (see, for instance, Theorem 2.5 of Akcoglu and Krengel \cite{AK:81}). For the sake of the reader we briefly recall  the argument next. 

Fix $x \in \R^n$, let   $\cI$ be the set of intervals of the form $[a,b) \subset [0,\infty)$, and consider the map $F: \cI \times \Om \to \R$
\begin{equation*}
F([a,b),\om) := \cL(bx,b,ax,a,\om).
\end{equation*}
Lemma \ref{lem:statsubadd} yields that  $F(\cdot,\om)$ is a stationary subadditive family with respect to the measure preserving semigroup $(\theta_{c})_{c \in \R_+}$ given by $\theta_c\, \om = \tau_{(cx,c)}\om$. Moreover, it follows from
 \eqref{eq:Lbdd}, that the family $\{F([a,b),\cdot)\,:\, [a,b) \subseteq (0,1)\}$ is uniformly integrable in $\Om$. 
 
Then the  subadditive ergodic theorem implies  the existence  a random variable $\ol{L}(x,\om; 0)$ which is invariant with respect to $\{\theta_c\}_{c\in\R_+}$, and an event $\Om_{x,0}$ with full measure, such that, for all  $\om \in \Om_{x,0},$
\begin{equation*}
\lim_{\rho \to \infty} \frac{1}{\rho} F([0,\rho),\om) = \lim_{\rho \to \infty} \frac{1}{\rho} \cL(\rho x, \rho, 0, 0, \om) = \ol{L}(x,\om; 0). 
\end{equation*}
Here, the parameter $0$ in $\ol{L}(x,\om;0)$ and $\Om_{x,0}$ indicates that the vertex of $\cL$ is $(0x,0)$. 

By the same argument, for each $k\in \Z$,  there exist a  random variable $\ol{L}(x,\cdot;k)$, which is  invariant with respect to $\{\theta_c\}_{c\in\R_+}$, and events $\Om_{x,k}$ of full measure such that, for all $\om \in \Om_{x,k}$,
\begin{equation}\label{eq:uc_k}
\lim_{\rho \to \infty} \frac{1}{\rho} \cL(\rho x, \rho, kx, k, \om) = \ol{L}(x,\om;k). 
\end{equation}

We note that,  for all $c \in \R_+$ and $k\in \Z$,  $\ol{L}(x, \theta_c \om;k) = \ol{L}(x,\om;k)$. Even so, $\ol{L}(x,\cdot;k)$ is not necessarily deterministic, because the semigroup $(\theta_c)_{c \in \R_+}$ may not be ergodic.

\medskip

{\itshape Step 2: The invariance of $\ol{L}(x,\cdot;k)$ with respect to $k \in \Z$.}  Let $\Om_x = \bigcap_{k\in \Z} \Om_{x,k}$. Then 
\begin{equation}
\label{eq:Lk}
\ol{L}(x,\cdot\,;k) = \ol{L}(x,\cdot\,;0) \quad \text{ on } \Om_x \  \text{ for all } k \in \Z.
\end{equation}
The $\{\theta_c\}$ invariance of  $\ol{L}(x,\cdot\,;k)$ and \eqref{e.L.stationary} imply, for all $\om \in \Om_x$ and $k \in \Z$,
\begin{equation*}
\ol{L}(x,\om;k) = \ol{L}(x,\tau_{(x,1)}\om;k) =  \lim_{\rho \to \infty} \frac{1}{\rho} \cL(\rho x + x, \rho +1, (k+1)x, k+1,\om).
\end{equation*}
The uniform H\"older continuity of $\frac{1}{\rho} \cL(\rho \cdot, \rho \cdot, (k+1)x,k+1,\om)$ in Proposition \ref{prop:uc} shows
\begin{equation*}
\lim_{\rho \to \infty} \frac{1}{\rho} \left|\cL(\rho x + x, \rho +1, (k+1)x, k+1,\om) - \cL(\rho x, \rho, (k+1)x, k+1,\om)\right| = 0.
\end{equation*}
Combining the last two observations, we find that
\begin{equation*}
\ol{L}(x,\cdot\,;k) = \lim_{\rho \to \infty} \frac{1}{\rho} \cL(\rho x, \rho, (k+1)x, k+1,\om) = \ol{L}(x,\cdot\,;k+1).
\end{equation*}

We henceforth denote $\ol{L}(x,\cdot;k)$ by $\ol{L}(x,\cdot)$, and conclude that the rescaled function $\rho^{-1}\cL(\rho x, \rho, kx, k, \om)$ converges to $\ol{L}(x,\cdot)$ for all $k\in \Z$ and $\om\in \Om_x$.
\medskip

{\itshape Step 3: $\ol{L}(x,\cdot)$ is deterministic}. We show that $\ol{L}(x,\cdot)$ is translation invariant with respect to $\{\tau_{(y,s)}\}$, $(y,s) \in \R^{n+1}$. The conclusion then follows from ergodicity of  $\{\tau_{(y,s)}\}$.

Fix $\om \in \Om_x$ and $(y,s) \in \R^{n+1}$ and choose $k_1 \in \Z$ so that $k_1 \in [s,s+1)$. It follows from  \eqref{e.L.stationary} and \eqref{mono-t}, that 
\begin{equation}
\label{eq:sae1}
\begin{aligned}
\cL(\rho x, \rho, 0,0,\tau_{(y,s)}\om) &= \cL(\rho x + y, \rho + s, y,s,\om)\\
 &\le  \cL(\rho x + y, \rho +s, k_1x, k_1,\om) + \cL(k_1 x, k_1, y,s,\om).
\end{aligned}
\end{equation}
Using \eqref{eq:Lbdd}, $k_1 - s \in [0,1)$, $\gam' \in (1,2)$ and that $k_1 x$ and $y$ are bounded, we observe
\begin{equation*}
\lim_{\rho \to \infty} \frac{1}{\rho} \left|\cL(k_1 x, k_1,y,s) \right| \le \lim_{\rho \to \infty} \frac{C}{\rho} \left( \left|k_1 x - y\right|^{\gam'} + |k_1 - s|^{1-\frac{\gam'}{2}} + |k_1 - s| \right) = 0.
\end{equation*}
For the other term in the right hand side of\eqref{eq:sae1}, we have
\begin{equation*}
\begin{aligned}
&\lim_{\rho \to \infty} \frac{1}{\rho} \cL(\rho x + y, \rho + s, k_1 x,k_1,\om) = \lim_{\rho \to \infty} \frac{1}{\rho} \cL(\rho x, \rho, k_1 x,k_1,\om)\\
&\qquad\qquad + \lim_{\rho \to \infty} \left[\frac{1}{\rho} \cL(\rho x + y, \rho + s, k_1 x,k_1,\om) - \frac{1}{\rho} \cL(\rho x, \rho, k_1 x,k_1,\om)\right].
\end{aligned}
\end{equation*}
As in Step 2, the second term on the right hand side above converges to zero in view of Proposition \ref{prop:uc}, while the limit of the first term is precisely $\ol{L}(x,\om)$. It follows that 
\begin{equation}
\label{eq:sae2}
\limsup_{\rho \to \infty} \frac{1}{\rho} \cL(\rho x, \rho, 0,0,\tau_{(y,s)}\om) \le \ol{L}(x,\om).
\end{equation}

Similarly, given $(y,s) \in \R^{n+1}$, we choose $k_2 \in \Z$ such that $k_2 \in (s-1,s]$ and argue as above to find
\begin{equation}
\label{eq:sae3}
\liminf_{\rho \to \infty} \frac{1}{\rho} \cL(\rho x, \rho, 0,0,\tau_{(y,s)}\om) \ge \ol{L}(x,\om).
\end{equation}
Since $(y,s)$ is arbitrary and $\Om_x$ has full measure, we conclude that $\ol{L}(x,\cdot)$ is translation invariant. 
\end{proof}

Next, we show that the limit $\ol{L}$ is local uniform continuous, and the convergence holds locally uniformly in $x$, again, with fixed vertices. 

\begin{lemma}\label{lem:Lext} Assume {\upshape(A)}. The map 
$\ol{L}: \R^n \to \R$ is locally uniformly continuous,
and 
and there exists an event $\Om_1$ with $\bP(\Om_1) = 1$ such that, for all $R>0$ and $\om \in \Om_1$,
\begin{equation}
\label{eq:Lx}
\lim_{\rho \to \infty} \sup_{x\in B_R} \left|\frac{1}{\rho} \cL(\rho x, \rho, 0, 0,\om) - \ol{L}(x)\right| = 0.
\end{equation}
\end{lemma}

\begin{proof} For any $R > 0$. For all $x,y \in B_R$, in view of Theorem \ref{thm:sae} and Proposition \ref{prop:uc}, there exist  $\om \in \Om_x \cap \Om_y$ such  that
\begin{equation*}
\ol{L}(x) - \ol{L}(y) = \lim_{\ep \to 0} \left( \cL^\ep(x,1,0,0,\om) - \cL^\ep(y,1,0,0,\om)\right) \le C|x-y|^\alpha,
\end{equation*}
where the H\"older component and the bound $C$ only depend on $R$ and the parameters in (A). Since the estimate above still holds if $x$ and $y$ are switched, it follows  that $\ol{L}$ is local uniform continuous.

For each $z \in \Q^n$, let $\Om_z$ be the event of full measure defined in Theorem \ref{thm:sae}. Let  $\Om_1 := \bigcap_{z \in \Q^n} \Om_z \in \cF$ and observe that $\bP(\Om_1) = 1$. 

Fix $R > 0$ . For any $x \in B_R$, there exist $\{x_k\} \in \Q^n \cap B_{2R}$ such that $x_k \to x$ as $k \to \infty$. Note that, for all $\om \in \Om_1$,
\begin{equation*}
\begin{aligned}
\left| \cL^\ep(x,1,0,0,\om) - \ol{L}(x)\right| \le &\left| \cL^\ep(x,1,0,0,\om) - \cL^\ep(x_k,1,0,0,\om)\right|\\
+ &\left| \cL^\ep(x_k,1,0,0) - \ol{L}(x_k)\right| + \left|\ol{L}(x_k)- \ol{L}(x)\right|.
\end{aligned}
\end{equation*}
Proposition \ref{prop:uc}, the fact that $\{x_k\}_{k\in\N} \cup \{x\} \subseteq B_{2R}$, and the local uniform continuity of $\ol{L}$ yield  that,  for all $\om \in \Om_1$, $\lim_{\ep \to 0} \cL^\ep(x,1,0,0,\om) = \ol{L}(x)$. It also follows from these facts that $\{\cL^\ep(\cdot,1,0,0,\om)\}_{\ep \in (0,1)}$ and $\ol{L}$ are equicontinuous on $B_{2R}$, and, hence, \eqref{eq:Lx} holds.
\end{proof}



Next, we prove Theorem \ref{t.L.homog}. 
The argument  follows as in \cite{AS2,AT2} from a combination of  Egoroff's and Birkhoff's ergodic theorems. We need, however, to extend the method to the setting of space-time random environment and, in particular, modify the reasoning so that is does not rely on uniform continuity with respect to the vertex.

\begin{proof}[Proof of Theorem \ref{t.L.homog}] {\itshape Step 1.} We first show that  \eqref{eq:Lx} to: for all $0 < r < R$ and $R \ge 1$,
\begin{equation}\label{e.unifLlim.1}
\bP\left[\lim_{\rho \to \infty} \sup_{(x,t) \in Q_{R,r,R}} \left| \frac{1}{\rho} \cL(\rho x, \rho t, 0,0,\om) - t\ol{L}\left(\frac{x}{t}\right) \right| = 0\right] = 1.
\end{equation}
Fix an $\om \in \Om_1$ and observe that
\begin{equation*}
\frac{1}{\rho} \cL(\rho x, \rho t, 0,0,\om) - t\ol{L}\left(\frac{x}{t}\right) = t\left[\frac{1}{\rho t} \cL\left(\frac{\rho t x}{t}, \rho t, 0,0,\om \right) - \ol{L}\left(\frac{x}{t}\right)\right].
\end{equation*}
Since $r \le t \le R$ and $|x| \le R$, we have $|x/t| \le R/r$, and
\begin{equation*}
\sup_{(x,t) \in Q_{R,r,R}} \left|\frac{1}{\rho t} \cL(\rho x, \rho t, 0,0,\om) - \ol{L}\left(\frac{x}{t}\right)\right| \le \sup_{\substack{r \le t \le R\\y \in B_{R/r}}} \left|\frac{1}{\rho t} \cL(\rho t y, \rho t, 0,0,\om) - \ol{L}(y)\right|.
\end{equation*}
In view of \eqref{eq:Lx}, for any $\delta > 0$, there exists $\rho_\del = \rho_\del(r,R,\om)> 0$ such that, if  $\rho' > \rho_\delta$, then 
\begin{equation*}
\sup_{y \in B_{R/r}} \left|\frac{1}{\rho'} \cL(\rho' y, \rho', 0,0,\om) - \ol{L}(y)\right| < \delta.
\end{equation*}
It follows that,  if $\rho > r^{-1} \rho_\del$, then  $\rho t > \rho_\delta$ for all $t \in [r,R]$ and, as a consequence,
\begin{equation*}
\sup_{r \le t \le R} \, \sup_{y \in B_{R/r}} \left|\frac{1}{\rho t} \cL(\rho t y, \rho t, 0,0,\om) - \ol{L}(y)\right| < \delta.
\end{equation*}
Combining the estimates above, yields  \eqref{e.unifLlim.1}.

\medskip

{\itshape Step 2}. We show that, for all $R > r > 0$ with $R \ge 1$, 
\begin{equation}\label{e.unifLlim.2}
\bP\left[\lim_{\rho \to \infty} \sup_{(y,s) \in Q_R} \sup_{(x,t)\in Q_{R,r,R}(y,s)} \left| \frac{1}{\rho} \cL(\rho x, \rho t, \rho y, \rho s,\om) - (t-s)\ol{L}\left(\frac{x-y}{t-s}\right) \right| = 0\right] = 1.
\end{equation}
Note that by choosing a sequence $R_k \uparrow \infty$, $r_k \downarrow 0$ and intersecting events of full measures, the above statement is equivalent to that of Theorem \ref{t.L.homog}. Hence, we only need to prove that 
\begin{equation}\label{eq:mv1}
\bP\left[\limsup_{\rho \to \infty} \sup_{(y,s) \in Q_R} \sup_{(x,t)\in Q_{R,r,R}(y,s)} \frac{1}{\rho} \cL(\rho x, \rho t, \rho y, \rho s,\om) - (t-s)\ol{L}\left(\frac{x-y}{t-s}\right) \le 0 \right] = 1,
\end{equation}
and
\begin{equation}\label{eq:mv2}
\bP\left[\liminf_{\rho \to \infty} \inf_{(y,s) \in Q_R} \inf_{(x,t)\in Q_{R,r,R}(y,s)} \frac{1}{\rho} \cL(\rho x, \rho t, \rho y, \rho s,\om) - (t-s)\ol{L}\left(\frac{x-y}{t-s}\right)  \ge 0\right] = 1.
\end{equation}

Observe that, in view of \eqref{e.unifLlim.1}, as $\rho \to \infty$  and for all $\om \in \Om_1$, 
\begin{equation*}
X_\rho(\om) := \sup_{(x,t) \in B_{R,r,R}} \left| \frac{1}{\rho} \cL(\rho x, \rho t, 0,0,\om) - t\ol{L}\left(\frac{x}{t}\right) \right| \to 0.
\end{equation*}
Then Egoroff's theorem yields, for any $0 < \ep <1$,  an event $\Om_\ep \subset \Om_1$ such that $\bP(\Om_\ep) \ge 1- \ep^{n+1}/8$ and
\begin{equation*}
\lim_{\rho \to \infty} \sup_{\om \in \Om_\ep} X_\rho(\om) = 0.
\end{equation*}
In particular, there exists $T_\ep > 0$ such that, for all $\rho > T_\ep$,
\begin{equation}\label{e.Egoroff}
\sup_{\om \in \Om_\ep} X_\rho(\om) < \frac{\ep}{2}.
\end{equation}
The ergodic theorem gives an event $\widetilde\Om_\ep$ such that $\bP(\widetilde \Om_\ep) = 1$ and for all $\om \in \widetilde\Om_\ep$,
\begin{equation*}
\lim_{K\to \infty} \frac{1}{\mathrm{Vol}(Q_K)} \int_{B_K} \int_{-K}^K \chi_{\Om_\ep} \left(\tau_{(y,s)} \om \right) ds dy = \bP(\Om_\ep) \ge 1- \frac{1}{8}\ep^{n+1}.
\end{equation*}
It follows that, for every $\om \in \widetilde\Om_\ep$, there exists $K_\ep(\om)$ such that if $K > K_\ep(\om)$,
\begin{equation*}
\mathrm{Vol } \left\{(y,s) \in Q_K \,:\, \tau_{(y,s)}\om \in \Om_\ep \right\} \ge \left(1-\frac{1}{4}\ep^{n+1}\right)\mathrm{Vol}(Q_K).
\end{equation*}

Let $\widetilde \Om_1$ be $\Om_1$,  for each $k \in \N$, $k \ge 2$, let $\widetilde\Om_{\frac 1 k}$ be defined as $\widetilde\Om_{\ep}$ with $\ep = {\frac 1 k}$,  set
set $\widetilde\Om: = \cap_{k=1}^\infty \widetilde\Om_{\frac{1}{k}},$
and note $\widetilde\Om \in \cF$ and $\bP(\widetilde\Om) = 1$. 

Fix now an $\om \in \widetilde\Om$. For any $\ep > 0$ small, choose $k$ large such that ${\frac 1 k} < \frac{\ep}{2}$, and,  for $R \ge 1$ given, set $\rho_\ep(\om) = R^{-1} \max\{T_{1/k}, K_{1/k}(\om)\}$, and observe that if $\rho > \rho_\ep$, then $\rho R > \max\{T_{1/k},K_{1/k}\}$.

For each $(y,s) \in Q_R$, let $C^+_{\rho \ep R}(y,s)$ (and, respectively, $C^-_{\rho \ep R}(y,s)$) be the region bounded between the cylinder $Q_{\rho \ep R}(y,s)$ and the cone at $(y,s)$ with unit upward (and, respectively, downward) opening, that is
\begin{equation*}
\begin{aligned}
C^+_{\rho \ep R}(y,s) := Q_{\rho \ep R}(y,s) \cap \{(x,t) \,:\, t > s, \, |x - y|/(t-s) \le 1\},\\
C^-_{\rho \ep R}(y,s) := Q_{\rho \ep R}(y,s) \cap \{(x,t) \,:\, t < s, \, |x - y|/(s-t) \le 1\},
\end{aligned}
\end{equation*}
and  note that,  for $\ep$ small,
\begin{equation*}
 \mathrm{Vol} \left( Q_{\rho R} \cap C^\pm_{\ep \rho R} \right) \, \ge \, \frac{1}{8}\ep^{n+1} \, \mathrm{Vol}(Q_{\rho R}).
\end{equation*}
It follows that, for every $(y,s) \in Q_{R}$, there exists $(\hat y, \hat s) \in Q_{R}$ such that $(\rho \hat{y}, \rho \hat{s}) \in C^+_{\rho \ep R}(\rho y, \rho s)$ and $\tau_{(\rho \hat y, \rho \hat s)}\om \in \Om_{1/k}$. 

We observe that
\begin{equation}
\label{eq:mv3}
\frac{1}{\rho} \cL(\rho x, \rho t, 0, 0, \tau_{(\rho y, \rho s)}\om) - t\ol{L}\left(\frac{x}{t}\right) = \frac{1}{\rho} \cL(\rho x, \rho t, 0, 0, \tau_{(\rho \hat{y},\rho \hat{s})}\om) - t\ol{L}\left(\frac{x}{t}\right) + E_\rho,
\end{equation}
with
\begin{equation*}
E_\rho :=\, \frac{1}{\rho} \cL(\rho(x+y), \rho(t+s), \rho y, \rho s,\om) - \frac{1}{\rho} \cL(\rho(x+\hat{y}), \rho(t+\hat{s}), \rho \hat{y}, \rho \hat{s},\om),
\end{equation*}
which is the error term resulted from the change of vertices. Because $(\rho \hat{y}, \rho \hat{s}) \in \Om_{1/k}$, the difference of the first two terms on the right hand side of \eqref{eq:mv3} is bounded from above by $\frac{\ep}{2}$.

In view of \eqref{mono-t}, the error $|E_\rho|$ can be bounded by
\begin{equation*}
|E_\rho| \leq |E_\rho ^1|+ |E_\rho^2|,
\end{equation*}
where 
\begin{equation*}
E_\rho ^1:= \frac{1}{\rho} \cL(\rho(x+y), \rho(t+s), \rho y, \rho s,\om) - \frac{1}{\rho} \cL(\rho(x+\hat{y}), \rho(t+\hat{s}), \rho y, \rho s,\om), 
\end{equation*}
and 
\begin{equation*}
E_\rho^2:=\frac{1}{\rho} \cL(\rho \hat{y}, \rho \hat{s}, \rho y, \rho s,\om).
\end{equation*}
Proposition \ref{prop:uc} yields that $|E_\rho ^1|=O(\ep^\alpha)$ for some exponent $\alpha$ depending on $R$, while   \eqref{eq:Lbdd} gives 
\begin{equation*}
|E_\rho^2| \le C\left( |s-\hat{s}| + \rho^{-\frac{\gam'}{2}} |s-\hat{s}|^{1-\frac{\gam'}{2}}\right) \le CR \ep,
\end{equation*}
provided that  $|y - \hat{y}|/|s-\hat{s}| \le 1$ and $|s - \hat{s}| \le \ep R$. 

In conclusion we have  that,  uniformly  for all $(y,s) \in Q_R$
\begin{equation*}
\frac{1}{\rho} \cL(\rho x, \rho t, 0, 0, \tau_{(\rho y, \rho s)}\om) - t\ol{L}\left(\frac{x}{t}\right) \le \frac{\ep}{2} + O(\ep^\alpha) + CR \ep, 
\end{equation*}
and, therefore, for all $\om \in \widetilde\Om$,
\begin{equation*}
\sup_{(y,s) \in Q_R} \, \sup_{(x,t) \in Q_{R,r,R}} \, \frac{1}{\rho} \cL(\rho x, \rho t, 0, 0, \tau_{(\rho y, \rho s)}\om) - t\ol{L}\left(\frac{x}{t}\right) \le \frac{\ep}{2} + \varrho(\ep R) + CR \ep.
\end{equation*}
Sending $\ep \to 0$, we obtain that
\begin{equation*}
\bP\left[\limsup_{\rho \to \infty} \sup_{(y,s) \in Q_R} \sup_{(x,t) \in Q_{R,r,R}} \frac{1}{\rho} \cL(\rho x, \rho t, 0,0,\tau_{(\rho y,\rho s)}\om) - t\ol{L}\left(\frac{x}{t}\right) \le 0\right] = 1.
\end{equation*}
In view of \eqref{e.L.stationary}, the statement above is equivalent to \eqref{eq:mv1}.

Similarly, by repeating the argument above, choosing $(\rho\hat{y},\rho\hat{s}) \in C^-_{\rho \ep R}(\rho y, \rho s)$ and $\tau_{(\rho \hat{y}, \rho \hat{s})} \om \in \Om_{1/k}$, we can bound the quantity in \eqref{eq:mv3} from below, and establish \eqref{eq:mv2}. 
\end{proof}

Finally, we note the following fact about $\ol{L}$ and $\ol{H}$ defined by \eqref{e.Hbar}. 

\begin{corollary} The functions $\ol{L} : \R^n \to \R$ and $\ol{H} : \R^n \to \R$ are convex.
\end{corollary}

The convexity of $\ol{L}$ is a straightforward consequence of Theorem \ref{t.L.homog}. Finally, as the Legendre transform of a convex function, $\ol{H}$ is also convex.

\section{The proof of Theorem \ref{thm:appcell}}
\label{s.lim=H}

According to the remarks at the end of Section \ref{sec:prelim}, this also completes the proof of the homogenization result of Theorem \ref{thm:stoch}.

For each $p \in \R^n$, let $w_\ep:= \ep w^\ep(\frac{\cdot}{\ep},\frac{\cdot}{\ep};\om,p)$, where $w^\ep$ is the solution to the approximate cell problem \eqref{eq:appcell}.   It follows  that 
\begin{equation}
\label{e.acell}
w_\ep + \left(w_\ep\right)_t - \ep\tr\left(A\left(\frac{\cdot}{\ep},\frac{\cdot}{\ep},\om\right) D^2 w_\ep\right) + H\left(p+Dw_\ep,\frac{\cdot}{\ep},\frac{\cdot}{\ep},\om\right) = 0 \quad \text{ in } \R^n \times \R.
\end{equation}
Then, for any $R>0$, \eqref{eq:weplim} is equivalent to
\begin{equation}
\label{eq:weplim2}
\limsup_{\ep \to 0} \sup_{(x,t) \in Q_R} \left| w_\ep(x,t;p,\om) + \ol{H}(p)\right| = 0.
\end{equation}

For the proof of  Theorem \ref{thm:appcell} we need to recall some notions from convex analysis. We have seen that $\ol{H}$ is a convex function defined on $\R^n$. The epigraph of $\ol{H}$ is defined by 
\begin{equation*}
\epi(\ol{H}) = \left\{(p,s)\,:\, p \in \R^n \ \text{and} \ s \in [\ol{H}(p),\infty) \right\}.
\end{equation*}
Note that $\epi(\ol{H})$ is a closed convex subset of $\R^{n+1}$. Given a closed convex subset $D$ of $\R^k$, a point $p \in D$ is called an extreme point if, whenever $p = \lambda x + (1-\lambda)y$, $x,y \in D$ and $\lambda \in [0,1]$, then either  $x = p$ or $y=p$. A point $p \in D$ is called an exposed point, if there exists a linear functional $f: \R^k \to \R$ such that $f(p) > f(p')$ for all $p' \in D \setminus \{p\}$.

We denote by $\partial \ol{L}(q)$ the sub-differential of $\ol{L}$ at $q$. If $\partial \ol{L}(q)$ contains exactly one element, then $\ol{L}$ is differentiable at $q$ and the unique element is $D\ol{L}(q)$. The following classification of vectors $p \in \R^n$ will be useful in the proof of Theorem \ref{thm:appcell}. 

\begin{lemma} \label{lem:epi} Let $\ol{L}$ and $\ol{H}$ be defined by Theorem \ref{thm:sae} and \eqref{e.Hbar} respectively. Then 

\noindent{\upshape(i)} for all $p \in \R^n$, $(p,\ol{H}(p))$ is on the boundary of $\epi(\ol{H})$ and $p \in \partial \ol{L}(q)$ for some $q \in \R^n$,
and

\noindent{\upshape(ii)} if $(p,\ol{H}(p))$ is an exposed point of $\epi(\ol{H})$, then $p = D\ol{L}(q)$ for some $q \in \R^n$.
\end{lemma}

\begin{proof} The domain of $\ol{H}$ is $\R^n$ and, since  $\ol{H}$ is continuous and locally bounded, it follows that $\ol{H}$ is a closed proper convex function. The first claim of part (i) is obvious. Hence, there exists $q \in \R^n$ so that the function $x \mapsto x \cdot q - \ol{H}(x)$ achieves its supremum at $p$. It follows that $q \in \partial \ol{H}(p)$. Since $\ol{H}$ is a closed proper convex function, by \cite[Corollary 23.5.1]{Rockafellar}, $p \in \partial \ol{L}(q)$ also holds. Part (ii) follows directly from \cite[Corollary 25.1.2]{Rockafellar}.
\end{proof}


\begin{proof}[Proof of Theorem \ref{thm:appcell}] {\it Step 1}: We prove that for any fixed $\om \in \widetilde\Om$, $p \in \R^n$ and $R \ge 1$,
\begin{equation}
\label{e.limsup<H}
\limsup_{\ep \to 0} \sup_{(x,t) \in Q_R} \left( w_\ep(x,t;p) + \ol{H}(p) \right) \le 0.
\end{equation}
Lemma \ref{lem:epi} (i) yields a $q \in \R^n$ such that $p \in \partial \ol{L}(q)$. This implies
\begin{equation}
\label{e.subgradientL}
\ol{H}(p) + \ol{L}(q) - p\cdot q = 0.
\end{equation}

Arguing by contradiction, we assume \eqref{e.limsup<H} fails, so there exist $\del > 0$, a subsequence  $\ep_k \to 0$, 
and a sequence $\{(z_k,s_k)\}_{k\in\N} \in Q_R$ such that
\begin{equation*}
w_{\ep_k}(z_k,s_k) + \ol{H}(p) \ge \del > 0.
\end{equation*}
For notational simplicity, the subscript $k$ in $\ep_k$ and in $(z_k,s_k)$ is henceforth suppressed. Since $\om$ and $p$ are also fixed, any dependence on these parameters is also suppressed. 

Next, for some small real number $c>0$ and some $\lam \in (0,1)$ close to $1$ to be chosen and  $(x,t) \in \R^n \times (-\infty,s)$, we define
\begin{equation*}
W^\ep(x,t) := \lam \left(w_\ep(x,t) - w_\ep(z,s)\right) - c\delta \psi(x) -c\del(s - t) ,
\end{equation*}
where 
\begin{equation*}
\psi(x):=\left((1+|x-z|^2)^{\frac 1 2} -1\right); 
\end{equation*}
note that 
\begin{equation*}
|D\psi(x)| < 1 \quad \text{and} \quad (1+|x|^2)^{-\frac 3 2} Id \le D^2 \psi(x) \le (1+|x|^2)^{-\frac 1 2} Id.
\end{equation*}
Let $U_\ep := \{(x,t) \in \R^n \times \R \,:\, W^\ep \ge -\frac{\del}{4}\} \cap \{(x,t) \in \R^n \times \R \,:\, t \le s\}$. It follows that 
\begin{equation}
\label{e.W.sub}
W^\ep_t - \ep\tr \left(A\left(\frac{x}{\ep},\frac{t}{\ep}\right) D^2 W^\ep\right) + H\left(p+DW^\ep,\frac{x}{\ep},\frac{t}{\ep}\right) \le \ol{H}(p) - \frac{\del}{4} \quad \text{in } U_\ep.
\end{equation}
Indeed, if $\varphi \in C^2(\R^n\times \R)$ and if $W^\ep - \varphi$ attains a local maximum at $(x_0,t_0)$ in $U_\ep$, then the mapping
\begin{equation*}
(x,t) \mapsto w_\ep(x,t) - \lam^{-1} (\varphi(x,t) + c\delta \psi(x-z) + c\del(s-t))
\end{equation*}
attains a local maximum at $(x_0,t_0)$.

Since $w_\ep$ is the viscosity solution of \eqref{e.acell}, we find
\begin{equation*}
\begin{aligned}
w_\ep(x_0,t_0) + \lam^{-1} \left(\varphi_t(x_0,t_0) - c\del\right) - \lam^{-1} \ep \tr\left(A\left(\frac{x_0}{\ep},\frac{t_0}{\ep}\right) (D^2 \varphi(x_0,t_0) + c\delta D^2 \psi(x_0)\right)\\
 + H\left(p + \lam^{-1} (D\varphi(x_0,t_0) + c\del D\psi(x_0)),\frac{x_0}{\ep},\frac{t_0}{\ep}\right) \le 0,
\end{aligned}
\end{equation*}
while the  convexity of $H$ in $p$ gives
\begin{equation*}
\begin{aligned}
&\ H\left(p+ D\varphi(x_0,t_0),\frac{x_0}{\ep},\frac{t_0}{\ep}\right)\\
= &\ H\Big(\lam\Big(p+\frac{D\varphi(x_0,t_0) + c\delta D\psi(x_0)}{\lam}\Big)+(1-\lam)\Big(p-\frac{c\del D\psi(x_0)}{1-\lam}\Big),\frac{x_0}{\ep},\frac{t_0}{\ep}\Big)\\
\le &\ \lam H\Big(p+\frac{D\varphi(x_0,t_0) + c\delta D\psi(x_0)}{\lam},\frac{x_0}{\ep},\frac{t_0}{\ep}\Big) + (1-\lam) H\Big(p-\frac{c\del D\psi(x_0)}{1-\lam},\frac{x_0}{\ep},\frac{t_0}{\ep}\Big).
\end{aligned}
\end{equation*}
We use the growth assumption $\eqref{A2}$ to choose $\lam(p) \in (0,1]$ so that $1 - \lam(p)$ is small and
\begin{equation*}
-\lam \del + \lam \ol{H}(p) + (1-\lam) \sup_{p' \in B_1(p)} \sup_{(x,t) \in \R^{n} \times \R} H(p',x,t) \le \ol{H}(p) - \frac{3\del}{4}, 
\end{equation*}
and we fix a small enough  $c>0$ so that $c < 1/8$ and $c\del<1-\lam$. Then, for all $(x,t) \in \R^n \times \R$,
\begin{equation*}
p - c\del(1-\lam)^{-1} D\psi(x) \in B_1(p) \quad\text{and}\quad |\tr(A(x,t) c\del D^2 \psi(x))| < \frac{\del}{16}.
\end{equation*}
Combining the estimates above, we get, for $\ep$ sufficiently small,
\begin{equation}
\label{eq:ineqH}
\begin{aligned}
&\varphi_t(x_0,t_0) - \ep\tr\left(A\left(\frac{x_0}{\ep},\frac{t_0}{\ep}\right) D^2 \varphi(x_0,t_0)\right) + H\left(p+D\varphi(x_0,t_0),\frac{x_0}{\ep},\frac{t_0}{\ep}\right)\\
\le\, &-\lam w_\ep(x_0,t_0) + c\del + \ep c\del \tr\left(A\left(\frac{x_0}{\ep},\frac{t_0}{\ep}\right)D^2\psi(x_0)\right) + H\left(p+D\varphi(x_0,t_0),\frac{x_0}{\ep},\frac{t_0}{\ep}\right) \\
& \quad \quad \quad \quad \quad - \lam H\left(p+\frac{D\varphi(x_0,t_0) + c\delta D\psi(x_0)}{\lam},\frac{x_0}{\ep},\frac{t_0}{\ep}\right)\\
\le\, &-W^\ep(x_0,t_0) - \lam w_\ep(z,s) + (1-\lam) H\left(p-\frac{c\del D\psi(x_0)}{1-\lam},\frac{x_0}{\ep},\frac{t_0}{\ep}\right) + \frac{\del}{4}\\
\le\, &-W^\ep(x_0,t_0) - \lam \del + \lam \ol{H}(p) + (1-\lam) \sup_{p' \in B_1(p)} \|H(p',\cdot,\cdot)\|_{L^\infty} + \frac{\del}{4}\\
\le\, &-W^\ep(x_0,t_0) + \ol{H}(p) - \frac{\del}{2} \le\, \ol{H}(p) - \frac{\del}{2}, 
\end{aligned}
\end{equation}
with the last inequality holding  because $(x_0,t_0) \in U_\ep$ and, hence, $-W^\ep(x_0,t_0) \le \frac{\del}{4}$. This proves \eqref{e.W.sub}.

Next we compare $W^\ep$ with $V^\ep := V^\ep(x,t)$ which is defined, for some large $r>0$ to be chosen, by
\begin{equation}
\label{e.V.def}
V^\ep(x,t) := \cL^\ep(x,t,z-rq,s-r) - \cL^\ep(z,s,z-rq,s-r) - p \cdot (x-z) + \ol{H}(p) (t-s).
\end{equation}
In view of \eqref{eq:cLe}, $V^\ep$ satisfies
\begin{equation*}
V^\ep_t - \ep\tr\left(A\left(\frac{x}{\ep},\frac{t}{\ep}\right)D^2 V^\ep\right) + H\left(p+DV^\ep,\frac{x}{\ep},\frac{t}{\ep}\right) = \ol{H}(p) \quad \text{ in } \R^n \times (-r+s,\infty).
\end{equation*}

Let $\partial_s U_\ep := \{t<s\}\cap \partial \{W^\ep \ge - \frac{\del}{4}\}$ be the parabolic boundary of the space-time domain $U_\ep$ and note that $W^\ep = - \frac{\del}{4}$ on $\partial_s U_\ep$. 

The comparison principle for \eqref{e.W.sub}, yields 
\begin{equation}
\label{eq:comp1}
\sup_{U_\ep} \left( W^\ep - V^\ep\right) = \sup_{\partial_s U_\ep} \left(W^\ep - V^\ep\right) = -\frac{\del}{4} - \inf_{\partial_s U_\ep} V^\ep,
\end{equation}
and, since $(z,s) \in U_\ep \cap \{t = s\}$ is an interior point of $U_\ep$ and $W^\ep(z,s) = V^\ep(z,s) = 0$, the left hand side is nonnegative. 

In view of the bound $|w_\ep| \le C$ and the linear growth of $\psi(x)+(s-t)$, we find that $U_\ep \subset Q_{R'}(z,s)$ provided $R' = 2C/c\del$. It follows that
\begin{equation*}
\begin{aligned}
\inf_{(x,t) \in Q_{R'}(z,s)}  \Big(\cL^\ep(x,t,z-rq,s-r) &- \cL^\ep(z,t,z-rq,s-r) \\
&- p\cdot(x-z) + \ol{H}(p)(t-s)\Big) \le  -\frac{\del}{4}.
\end{aligned}
\end{equation*}

Send $\ep \to 0$. Since $\{(z_j,s_j)\} \subset Q_R$, we may assume that $(z,s) \to (z_0,s_0)$. By Theorem \ref{t.L.homog}, $\cL^\ep$ converges uniformly on $Q_{R'}(z,s)$. We get
\begin{equation*}
\begin{aligned}
\inf_{(x,t) \in Q_{R'+1}(z_0,s_0)} \Big( (r-s_0+t)\ol{L}\left(\frac{x-z_0+rq}{r-s_0+t}\right) &- r\ol{L}(q) \\
- p\cdot(x-z_0) &+ \ol{H}(p)(t-s_0) \Big) \le -\frac{\del}{4}.
\end{aligned}
\end{equation*}
The fact $p \in \partial \ol{L}(q)$ implies
\begin{equation}
\label{e.subgradientL2}
\ol{L}\left(\frac{x-z_0+rq}{r-s_0+t}\right) - \ol{L}(q) \ge p \cdot \left(\frac{x-z_0+rq}{r-s_0+t} - q\right) = p \cdot \frac{(x-z_0) + (s_0-t)q}{r-s_0+t}.
\end{equation}
As a result, for $r$ sufficiently large, we have
\begin{equation*}
\inf_{(x,z) \in Q_{R'+1}(z_0,s_0)} \left( (t-s_0)\left(\ol{H}(p) + \ol{L}(q) - p\cdot q\right) \right) \le -\frac{\del}{4},
\end{equation*}
which, combined with \eqref{e.subgradientL}, yields $0 \le -\del/4$. This is a contradiction and, hence, \eqref{e.limsup<H} must hold.
\bigskip

{\it Step 2}: For any fixed $\om \in \widetilde\Om$, $p \in \R^n$ and $R \ge 1$,
\begin{equation}
\label{e.liminf>H}
\liminf_{\ep \to 0} \inf_{(x,t) \in Q_R} w_\ep(x,t;p) + \ol{H}(p) \ge 0.
\end{equation}
We claim that this task can be reduced to the case of $(p,\ol{H}(p))$ being an exposed point of $\epi(\ol{H})$.

Indeed, assume that \eqref{e.liminf>H} holds for all exposed $(p,\ol{H}(p))$. Then if $p \in \R^n$ is such that $(p,\ol{H}(p))$ is an extreme point of $\epi(\ol{H})$, then by Straszewicz's theorem \cite[Theorem 18.6]{Rockafellar}, there exists a sequence $\{p_j\}$ converging to $p$, such that $\{(p_j,\ol{H}(p_j))\}$ are exposed points of  $\epi(\ol{H})$. In view of the continuity of the mapping $p \mapsto w^\ep(\cdot,\cdot,p)$, \eqref{e.liminf>H} holds for extremal $(p,\ol{H}(p))$.

For any other $p \in \R^n$, $(p,\ol{H}(p))$ can be written as a convex combination of extremal $\{(p_j,\ol{H}(p_j))\}_{j=1}^{n+2}$. We have proved that \eqref{e.liminf>H} holds for each $p_j$. Since the mapping $p \mapsto w^\ep(\cdot,\cdot,p)$ is concave, and $p$ is a convex combination of $\{p_j\}_{j=1}^{n+2}$, we conclude that \eqref{e.liminf>H} holds for $p$.

{\it Step 3}: If $p \in \R^n$ and if $(p,\ol{H}(p))$ is an exposed point of $\epi(\ol{H})$,  then  \eqref{e.liminf>H} holds.
Although the proof of \eqref{e.liminf>H} follows along the lines of Step 1, there is an important difference. The inequality \eqref{e.subgradientL}, which holds for any $p \in \partial \ol{L}(q)$, is useful only to establish the upper bound as seen in Step 1. Here, however, $p$ satisfies the additional condition that $(p,\ol{H}(p))$ is exposed, and, hence, in view of  Lemma \ref{lem:epi}, $p = D\ol{L}(q)$ for some $q \in \R^n$. This amounts to
\begin{equation}
\label{e.gradientL}
\ol{L}(y) - \ol{L}(q) = p \cdot (y-q) + o(|y-q|),
\end{equation}
which is a stronger fact than \eqref{e.subgradientL}.

Arguing by contradiction, we assume that \eqref{e.liminf>H} fails, so there exists $\del > 0$, a subsequence $\{\ep_k\}_{k\in\N}$ converging to $0$, a sequence $\{(z_k,s_k)\}_{k\in\N} \subseteq Q_R$ such that
\begin{equation*}
-w_{\ep_k}(z_k,s_k) - \ol{H}(p) \ge \del > 0;
\end{equation*}
as before, the subscript $k$ is suppressed henceforth. 

Using  \eqref{A2}, we take  $\lam>1$ such that
\begin{equation*}
\lam \del + \lam \ol{H}(p) + (\lam-1) \inf_{p' \in B_1(p)} \inf_{(x,t) \in \R^{n} \times \R} H(p',x,t) \ge \ol{H}(p) + \frac{3\del}{4}.
\end{equation*}
After $\lam$ is fixed, we choose $0 < c < \frac{1}{8}$ so that $c\del < \lam-1$, and for $x \in \R^n$ and $t \le s$, we define
\begin{equation*}
W^\ep(x,t) := \lam \left(w_\ep(x,t) - w_\ep(z,s)\right) + c\del \left((1+|x-z|^2)^{\frac 1 2} -1\right) + c\del(s - t),
\end{equation*}
and set $U_\ep := \{W^\ep \le \frac{\del}{4}\} \cap \{t\le s\}$. 

We claim that 
\begin{equation}
\label{e.W.sup}
W^\ep_t - \ep\tr \left(A\left(\frac{x}{\ep},\frac{t}{\ep}\right) D^2 W^\ep\right) + H\left(p+DW^\ep,\frac{x}{\ep},\frac{t}{\ep}\right) \ge \ol{H}(p) + \frac{\del}{4} \quad \text{in } U_\ep.
\end{equation}
This can be proved by the same argument that led to \eqref{e.W.sub}, provided we replace \eqref{eq:ineqH} by
\begin{equation*}
\begin{aligned}
H\left(p+D \varphi(x_0,t_0),\frac{x_0}{\ep},\frac{t_0}{\ep}\right) \,-\, &\lam H\left(p+\frac{D\varphi(x_0,t_0)-c\del D\psi(x_0)}{\lam},\frac{x_0}{\ep},\frac{t_0}{\ep}\right) \\
\ge &\, (\lam-1) H\left(p- \frac{c\del D \psi(x_0)}{\lam-1},\frac{x_0}{\ep},\frac{t_0}{\ep}\right).
\end{aligned}
\end{equation*}

Then we compare $W^\ep$ with the function $V^\ep$ defined by \eqref{e.V.def} on the domain $U_\ep$, and get
\begin{equation*}
\sup_{U_\ep} \left(V^\ep - W^\ep\right) = \sup_{\partial_s U_\ep} \left(V^\ep - W^\ep\right) = -\frac{\del}{4} + \sup_{\partial_s U_\ep} V^\ep,
\end{equation*}
The left hand side is non-negative since $V^\ep(z,s) = W^\ep(z,s) = 0$ and $(z,s)$ is an interior point of $U_\ep$. Moreover, if $R' = 2\|w_\ep\|_{L^\infty}/c\delta$, then $U_\ep \subset Q_{R'}(z,s)$, and, hence
\begin{equation*}
\begin{aligned}
\sup_{Q_{R'}(z,s)} \Big(\cL^\ep(x,t,z-rq,s - r) &- \cL^\ep(z,t,z-rq,s-r)\\
&- p\cdot(x-z) + \ol{H}(p)(t-s)\Big) \ge \frac{\del}{4}.
\end{aligned}
\end{equation*}
As in Step 1, we may assume $(z_k,s_k) \to (z_0,s_0) \in \ol{Q}_R$. Sending $\ep_k$ to $0$, we get
\begin{equation}
\label{e.liminf>H1}
\begin{aligned}
\sup_{(x,t) \in Q_{R'+1}(z_0,s_0)} \Big( (r-s_0+t)\ol{L}\left(\frac{x-z_0+rq}{r-s_0+t}\right) &- r\ol{L}(q)  \\
- p\cdot(x-z_0) &+ \ol{H}(p)(t-s_0) \Big) \ge \frac{\del}{4}.
\end{aligned}
\end{equation}

Using  that $p = D\ol{L}(q)$ and $\ol{L}(q) + \ol{H}(p) = p\cdot q$, we have
\begin{equation*}
\begin{aligned}
&\,(r-s_0+t)\ol{L}\left(\frac{x-z_0+rq}{r-s_0+t}\right) - r\ol{L}(q)  - p\cdot(x-z_0) + \ol{H}(p)(t-s_0)\\
= &\,(r-s_0+t) \left[ \ol{L}\left(\frac{x-z_0+rq}{r-s_0+t}\right) - \ol{L}(q) - p\cdot \frac{x-z_0 +rq -(r - s_0+t)q}{r-s_0+t} \right]\\
= &\, (r-s_0+t) \cdot o\left(\left|\frac{x-z_0+(s_0-t)q}{r-s_0+t}\right|\right)
\end{aligned}
\end{equation*}
Since $|x-z_0 + (s_0-t)q| \le (1+|q|)R$ is finite and the estimate \eqref{e.liminf>H1} holds for all large $r$, sending $r \to \infty$,  yields $\frac{\del}{4} \le 0$, which is a contradiction. 
\end{proof}


\section{Some Formulae for the Effective Hamiltonian}
\label{sec:formulae}

Arguments similar to the ones in \cite{LS2} yield that, once homogenization theory is established, the effective Hamiltonian $\ol{H}(p)$ 
is given by 
\begin{equation*}
\ol{H}(p) = \inf_{\psi \in \cS} \sup_{(x,t) \in \R^{n+1}} \left[  \psi_t - \tr\left( A(x,t) D^2 \psi(x,t) \right) + H(p+D\psi(x,t),x,t)\right],
\end{equation*}
where  the sup of the value of the differential operator evaluated on $\psi$ should be interpreted in the viscosity sense, and 
\begin{equation*}
\begin{aligned}
\cS := \Big\{ \psi: \R^{n+1} \times \Om \to \R ~:~ \psi(\cdot,\cdot,\om) \in C(\R^{n+1}), \\
\, \lim_{|(x,t)| \to \infty} \frac{|\psi(x,t,\om)|}{|(x,t)|} = 0 \text{ for a.s. } \om \in \Om,\\
\psi(x+y,t+s,\om) - \psi(x,t,\om) \text{ is stationary in } (y,s) \text{ for all } (x,t) \in \R^{n+1}\Big\}.
\end{aligned}
\end{equation*}
that is, $\cS$ is the set of random processes that are sublinear in $(x,t)$ and have stationary increments. Note that if $\psi \in \cS$ is also differentiable with respect to $(x,t)$, then the stationarity of increments is equivalent to $\psi_t$ and $D\psi$ being stationary, and the sublinearity is equivalent to $\E [\psi_t] = 0$ and $\E [D\psi] = 0$. 

Another formula for effective Hamiltonian was introduced in \cite{KRV} for time homogeneous random environment, and then generalized in \cite{KV} to space-time random environment, both under the assumption that the diffusion term is given by the identity matrix. We recall how this formula was obtained, and write it in the form that it should take when the diffusion matrix is more general. 

Any random variable $\widetilde{b}$ gives rise to a stationary random process $b(x,t,\om) = \widetilde{b}(\tau_{(x,t)}\om)$. In the reverse direction, for any stationary random process $b(x,t,\om)$, we can lift it to the probability space and identify it with $\widetilde b(\om) := b(0,0,\om)$. For notational simplicity, we omit the tilde in $\widetilde{b}$ from now on. The translation group $\{\tau_{(x,t)}: (x,t) \in \R^{n+1}\}$ acting on $L^2(\Om)$ are isometric. Let $\partial_t, D_i$, $i = 1,2,\cdots,n$, by an abuse of notations, be the corresponding infinitesimal generators; we denote further $D = (D_1,\cdots,D_n)$. 

Let $\bB := L^\infty(\Om,\R^n)$ be the space of essentially bounded maps from $\Om$ to $\R^n$. Given any $b \in \bB$ and $A = \sigma \sigma^T$ satisfying (A1), (A2) and (A3), let $x(t,\om)$ be the diffusion process starting from $0$ at time $0$ such that
\begin{equation*}
dx(t) = b(\tau_{(x(t),-t)}\om) dt + \sqrt{2} \sigma(x(t),-t) dB_t \quad \text{ for all } t > 0.
\end{equation*}
In the above, $(B_t)_{t\ge 0}$ is a standard $m$-dimensional Brownian motion, independent of $H$ and $\sigma$. This  process can be viewed as a diffusion  in the probability space as follows. Pick a starting point $\omega \in \Omega$, and define the walk $\omega(t) = \tau_{(x(t,\om),-t)} \omega$, $t\ge 0$. This is a Markov process on $\Omega$ with generator
\begin{equation}
    \mathcal{L}_{b,\sigma} = -\partial_t + \tr(\sigma(\omega)\sigma(\omega)^T D^2) + b(\omega) \cdot D.
    \label{e.pdiff}
\end{equation}

Let $\bD := \{\Phi \in L^\infty(\Om;\R) \,:\, \E[\Phi] = 1, \Phi > 0 \text{ and } (\partial_t \Phi, D\Phi) \in L^\infty\}$. Finally, let
\begin{equation}
  \cE := \left\{(b,\Phi) \in \bB \times \bD \;:\; \partial_t \Phi + D^2_{ij}(A\Phi) - \Div (b\Phi) = 0\right\},
  \label{e.invmeas}
\end{equation}
where the equation should be understood in the weak sense, that is for all $G \in C^\infty_0(\R^{n+1},\R)$,
\begin{equation*}
\begin{aligned}
	\int_{\R} \int_{\R^n} \big[\partial_t G(x,t) &+ \langle -b + \Div A, DG(x,t)\rangle\big] \Phi (\tau_{(t,x)}\omega) \\
	& + \langle A D\Phi(\tau_{(t,x)}\omega), DG(x,t)\rangle \ dx dt = 0.
\end{aligned}
\end{equation*}
Hence, $\cE$ consists of all pairs $(b,\Phi)$ such that $\Phi$ is the density of an invariant measure of the Markov process $\mathcal{L}_{b,\sigma}$. We note that for any $v \in \R^n$, the pair $(b,\Phi)$, where $b_j = v_j + D_i A_{ij}$ and $\Phi \equiv 1$, satisfies the equation above and, hence, $\cE$ is nonempty. 

Following \cite{KRV,KV}, the effective Hamiltonian, for each $p \in \R^n$, should be given by
\begin{equation}
	\widetilde{H}(p) = \sup_{(b,\Phi)\in \cE} \E \left[\left( \langle -b, p\rangle - L(-b(\omega),\omega)\right) \Phi(\omega) \right].\\
  \label{e.H_KRV}
\end{equation}
Note that in this formula, $A$ does not need to be uniformly elliptic and can be degenerate.

\medskip

As a corollary of Theorem \ref{thm:stoch}, we can show that the above formulae for effective Hamiltonian holds in the setting of this paper. 

\begin{theorem} Assume {\upshape(A)} so that Theorem \ref{thm:stoch} holds. Then, for all $p \in \R^n$, $\ol{H}(p) = \widetilde H(p)$. 
\end{theorem}

We only sketch the proof. Given the homogenization result, Theorem B of \cite{LS3} provides a method to establish $\widetilde{H} \le \ol{H}$, which is easily applied here. Note that even though \cite{LS3} concerned only time homogeneous environment, the proof of Theorem B there does not rely on this fact.

The inequality $\widetilde{H}(p) \ge \ol{H}(p)$ follows from the fact that, for any $\del > 0$, there exists $\psi_\del \in \cS$, such that $\psi_\del$ is a subsolution to
\begin{equation*}
\begin{aligned}
 \partial_t \psi_\del - \tr(A(x,t,\omega)D^2 \psi_\delta) + H(p+D\psi_\delta,x,t,\omega) \le \widetilde{H}(p) + \delta \quad \text{ on } \,  \R^{d+1}.
 \end{aligned}
\end{equation*}
This claim is proved in \cite{KV} for $A \equiv Id$, but the proof, which utilizes the min-max theorem, extends easily to general diffusion matrix $A \in C^{1,\alpha}$. We emphasize that neither $A = Id$ nor $A$ being uniformly elliptic is needed for this claim. 

It is difficult to prove the homogenization result of this paper using the method of \cite{KRV,KV}. Indeed, in these references, the uniform lower bound $\liminf_{\ep \to 0} \inf_{Q_R}(u^\ep - u) \ge 0$ is established using the ergodic theorem, which requires uniqueness of invariant measure for a given drift. For this, the uniform ellipticity of $A$ is crucial. The stronger assumption that $H$ grows superquadratically in $p$ does not seem to help to remove uniform ellipticity requirement of $A$. In that sense, the fact that \eqref{e.H_KRV} provides the formula for the effective Hamiltonian for possibly degenerate diffusion matrix $A$, though only under the restrictive superquadratic growth assumption, is a new fact.


\bibliographystyle{acm} 
\bibliography{jst2}

\section*{Acknowledgements}
WJ is supported in part by the NSF grant DMS-1515150. PS is supported in part by the NSF grant DMS-1266383 and DMS-1600129. HT is supported in part by the NSF grant DMS-1361236.

\bigskip

\end{document}